\documentclass[12pt]{amsart}

\usepackage[english]{babel}
\usepackage{amssymb,latexsym}
\usepackage{amsmath}
\usepackage{enumerate}

\newtheorem{thm}{Theorem}

\newtheorem{lem}{Lemma}

\begin{document}

\title{Preservation of complete Baireness}
\author{Sergey Medvedev}
\address{South Ural State University,  Chelyabinsk, Russia}
\email{medvedevsv@susu.ru}

\begin{abstract}
The main result is the following.
Let $f \colon X \rightarrow Y$ be a continuous mapping of a completely Baire space $X$ onto a hereditary weakly Preiss-Simon regular space $Y$ such that the image of every open subset of $X$ is a resolvable set in $Y$. Then $Y$ is completely Baire.

The classical Hurewicz theorem about closed embedding of the space of rational numbers into metrizable spaces is generalized to weakly Preiss-Simon regular spaces.
\end{abstract}

\subjclass[2010]{54H05, 03E15}

\keywords
{resolvable set, completely Baire space, weakly Preiss-Simon space, rational numbers}

\maketitle

One of the main tasks considered in topology is to find the properties of topological spaces that are preserved when moving to the image of the space with this property. For example, Sierpi\'{n}ski~\cite{Si} and Vain\v{s}tain~\cite{Va} have established that if $Y$ is an image of a completely metrizable space $X$ under a closed mapping or an open mapping, then $Y$ is completely metrizable.
A natural question to ask is which maps preserve complete metrizability.

Ostrovsky~\cite{Ost} proved that if a continuous mapping from a zero-dimensional Polish  (separable, completely metrizable) space $X$ onto a metrizable space $Y$ takes any clopen set $W \subset X$ to the union of an open set and a closed set, then $Y$ is completely metrizable. He raised the question whether the same is true when the images are the intersection of an open set and a closed set.
Gao and Kieftenbeld~\cite{GK} showed that if the image under $f \colon X \rightarrow Y$ of every open set or every closed set in a Polish space $X$ is a resolvable set in $Y$, then $Y$ is Polish. Holick\'{y} and Pol~\cite{HP} established similar results for the nonseparable case. The latter result solves Ostrovsky's problem for metrizable spaces.

It is notorious~\cite[Theorem 4.3.26]{Eng} that a topological space is completely metrizable if and only if it is a \v{C}ech-complete metrizable space.
Holick\'{y}~\cite[Corollary 3.3]{Hol} proved among other things that if $f \colon X \rightarrow Y$ is a continuous mapping from a \v{C}ech-complete space $X$ onto a metrizable space $Y$ such that $f$ takes open sets in $X$ to resolvable sets in $Y$, then $Y$ is completely metrizable. 

Let us pass now from a \v{C}ech-complete metrizable space to a completely Baire space.
Suppose we are given a continuous mapping  $f \colon X \rightarrow Y$  from a zero-dimensional  metrizable space $X$ onto a metrizable space $Y$ such that $f$ takes clopen sets in $X$ to resolvable sets in $Y$. Holick\'{y} and Pol~\cite[Corollary~4]{HP} showed that if $X$ is completely Baire, then so is $Y$. 
Our main Theorem~\ref{Th-m} generalizes this statement to regular spaces.
In order to obtain it, we  strengthen the Hurewicz theorem~\cite{Hur} about closed embedding of the space of rational numbers into metrizable spaces, see Theorem~\ref{t:1}.

\textbf{Notation.} For all undefined terms see~\cite{Eng}.

A set $E$ in a space $X$ is said to be \textit{resolvable}~\cite[\S~12]{Kur1} if it can be represented as
\begin{equation*}
E= (F_1 \setminus F_2) \cup (F_3 \setminus F_4) \cup \ldots \cup (F_\xi \setminus F_{\xi+1}) \cup  \ldots ,
\end{equation*}
where $\langle F_\xi \rangle$ forms a decreasing transfinite sequence of closed sets in $X$. Equivalently, $E$ is resolvable if, for every closed set $F \subset X$, the boundary of $F \bigcap E$ relative to $F$ is nowhere dense in $F$. Resolvable sets form a field of subsets of $X$ which contains closed sets.

A topological space $X$ is called a \textit{Baire space} if the intersection of countably many dense open sets in $X$ is dense. We call a space \textit{completely Baire} (or hereditarily Baire cf.~\cite{HP}) if every closed subspace of it is Baire. By Hurewicz’s theorem~\cite{Hur}, a metrizable space is completely Baire if and only if it contains no closed homeomorphic copy of rational numbers.

A topological space $X$ is called \textit{Preiss-Simon} at a point $x \in X$ if for any subset $A \subset X$ with $x \in  \overline{A}$ there is a sequence $\langle U_n  \colon n \in \omega \rangle$ of nonempty open subsets of $A$ that converges to $x$ in the sense that each neighborhood of $x$ contains all but finitely many sets $U_n$. A space $X$ is called a \textit{Preiss-Simon space} if $X$ is Preiss-Simon at each point $x \in X$. It is clear that every first countable space is Preiss-Simon and each Preiss–Simon space is Fr\`{e}chet–Urysohn.

We shall say that a space $X$ is \textit{weakly Preiss-Simon} at a point $x \in X$ if there is a sequence $\langle U_n  \colon n \in \omega \rangle$ of nonempty open subsets of $X$ that converges to $x$. Denote by $wPS(X)$ the set of points $x \in X$ at which $X$ is weakly Preiss-Simon. A space $X$ is called \textit{weakly Preiss-Simon} space if the set $wPS(X)$ is dense in $X$. A space $X$ is called \textit{hereditary weakly Preiss-Simon} space if every nonempty closed subset of $X$ is a weakly Preiss-Simon space, i. e., the set $wPS(F)$ is dense in $F$ for every nonempty closed set $F \subset X$. Clearly, every Preiss-Simon space is hereditary weakly Preiss-Simon.

\begin{lem}\label{L:1}
Let $F$ be a nowhere dense closed subset of a regular space $X$. Suppose $X$ is weakly Preiss-Simon at a non-isolated point $x$. Then for every neighborhood $O$ of $x$ there is a sequence $\langle U_n  \colon n \in \omega \rangle$ of nonempty open sets such that{\rm{:}}

\begin{enumerate}[\upshape (a)]

\item $\langle U_n  \colon n \in \omega \rangle$ converges to $x$,

\item $\overline{U_n} \subset O \setminus \{x \}$ and $\overline{U_n} \bigcap F = \emptyset$  for every $n \in \omega$,

\item $\overline{U_n} \bigcap \overline{U_k} = \emptyset$ provided $n \neq k$,

\item $\overline{\bigcup \{U_n \colon n \in \omega \} } = \{x \} \cup \bigcup \{ \overline{U_n} \colon n \in \omega \}$.
\end{enumerate}
\end{lem}

\begin{proof}
Take a sequence $\langle V_k  \colon k \in \omega \rangle$ of nonempty open sets in $X$ that converges to $x$. The neighborhood $O$ of $x$ contains all but finitely many sets $V_k$. Choose the least number $k_0$ with $V_{k_0} \subset O$. The point $x$ is non-isolated in $X$. Then there is a point $x_0 \in V_{k_0} \setminus (\{x \} \bigcup F)$ and its neighborhood $U_0$ satisfying $\overline{U_0} \subset V_{k_0} \setminus (\{x \} \bigcup F)$. Similarly, choose the least $k_1$ with $V_{k_1} \subset O \setminus \overline{U_0}$. Find a point $x_1 \in V_{k_1} \setminus \{x \}$ and its neighborhood $U_1$ satisfying $\overline{U_1} \subset V_{k_1} \setminus (\{x \} \bigcup F)$. Clearly, $\overline{U_0} \bigcap \overline{U_1}  = \emptyset$. Next, choose the least $k_2$ with $V_{k_2} \subset O \setminus ( \overline{U_0} \bigcup \overline{U_1})$. Find a point $x_2 \in V_{k_2} \setminus \{x \}$ and its neighborhood $U_2$ satisfying $\overline{U_2} \subset V_{k_2} \setminus (\{x \} \bigcup F)$, and so on. One can check that all conditions (a)--(d) hold.
\end{proof}

The following statement generalizes the Hurewicz theorem~\cite{Hur} about closed embedding of the space of rational numbers $\mathbb{Q}$ into metrizable spaces. Note that Theorem~\ref{t:1} was proved in \cite{M86} for a regular space $X$ of the first category with a dense first-countable subset.

\begin{thm}\label{t:1}
Let $X$ be a regular space of the first category such that the set $wPS(X)$ is dense in $X$. Then $X$ contains a closed copy of the space of rational numbers.
\end{thm}

\begin{proof}
Let $X = \bigcup \{X_n \colon n \in \omega \}$, where each $X_n$ is a closed nowhere dense subset of $X$. Without loss of generality, $X_0 \subset X_1 \subset \ldots$.

By induction on the tree $\omega^{< \omega}$, we shall construct points $x_t \in wPS(X)$, open sets $U_t$, and sequences $\langle U_{t \hat{\,} k} \colon k \in \omega \rangle$ of open sets such that the following conditions hold for each $t \in \omega^{< \omega}$:

\begin{enumerate}[\upshape (a)]

\item $U_t$ is a neighborhood of the point $x_t$,

\item $\overline{U_t} \bigcap X_n = \emptyset$ whenever $t \in \omega^{n+1}$,

\item $\overline{U_{t \hat{\,} k}} \subset U_t \setminus \{x_t \}$ for each $k \in \omega$,

\item the sequence $\langle U_{t \hat{\,} k} \colon k \in \omega \rangle$  converges to $x_t$,

\item  $\overline{U_{t \hat{\,} i}} \bigcap \overline{U_{t \hat{\,} k}} = \emptyset$ provided $i \neq k$,

\item $\overline{\bigcup \{U_{t \hat{\,} k} \colon k \in \omega \} } = \{x_t \} \cup \bigcup \{ \overline{U_{t \hat{\,} k}} \colon k \in \omega \}$.
\end{enumerate}

The space $X$ has no isolated points as a set of the first category.

For the base step of the induction, take a point $x_\emptyset \in wPS(X)$ and $U_\emptyset = X$. By Lemma~\ref{L:1} with $F = X_0$ and $O = X$, there is a sequence $\langle U_k  \colon k \in \omega \rangle$ of nonempty open sets satisfying conditions (a)--(f) for $n =0$.

Suppose the points $x_s$ and the sequences $\langle U_{s \hat{\,} k} \colon k \in \omega \rangle$ of open sets have been constructed for every $s \in \omega^j$ with $j \leq n$. Fix $s \in \omega^n$. The set $X_{n+1}$ is nowhere dense in the open set $U_{s \hat{\,} i}$ for every $i \in \omega$. Then we can find a point $x_{s \hat{\,} i} \in wPS(X) \bigcap (U_{s \hat{\,} i} \setminus X_{n+1})$. By Lemma~\ref{L:1} with $F = X_{n+1}$ and $O = U_{s \hat{\,} i}$, there is a sequence $\langle U_{s \hat{\,} i \hat{\,} k}  \colon k \in \omega \rangle$ of nonempty open sets satisfying conditions (a)--(f) for $t = s \hat{\,} i$. This completes the induction step.

Clearly, the set $Z = \bigcup \{x_t \colon t \in \omega^{< \omega}\}$ is countable and has no isolated points. For every $t \in \omega^{< \omega}$ from conditions (d) and (e) it follows that the sequence $\langle V_{t \hat{\,} n} \colon n \in \omega \rangle$ forms a countable base at the point $x_t$ with respect to $Z$, where 
\[V_{t \hat{\,} n} = Z \bigcap \left(U_t \setminus \cup \{\overline{U_{t \hat{\,} k}} \colon k <n \} \right).\]
Hence, $Z$ is a second-countable space. The Urysohn theorem \cite[Theorem 4.2.9]{Eng} implies that $Z$ is metrizable. According to the Sierpi\'{n}ski theorem~\cite{Si} (see also \cite[Exercise 6.2.A]{Eng}), $Z$ is homeomorphic to the space of rational numbers.

It remains to verify that $Z$ is closed in $X$. For each $n \in \omega$ define the sets
\begin{align*}
Z_n &= \bigcup \left\{x_s \colon s \in \omega^j, j \leq n \right\}, \\
P_n &= \bigcup \left\{\overline{U_{t \hat{\,} k}} \colon t \in \omega^n, k  \in \omega \right\}. 
\end{align*}
Then $Z = \bigcup \{Z_n \colon n \in \omega \}$, $Z_n \cap P_n = \emptyset$, and $Z \subset Z_n \bigcup P_n$ for each $n$.

By induction on $n \in \omega$, let us check that the union $Z_n \cup P_n$ is closed in $X$ for each $n$. 

For $n = 0$, the set $Z_0 \cup P_0$ is closed in $X$ by Lemma~\ref{L:1}. We now show that $Z_{n+1} \cup P_{n+1}$ is closed under the induction assumption on closeness of $Z_n \cup P_n$. Suppose that, on the contrary, there is a point 
\[x \in \overline{(Z_{n+1} \cup P_{n+1})} \setminus (Z_{n+1} \cup P_{n+1}).\] 
Since $(Z_{n+1} \setminus Z_n) \cup P_{n+1} \subset P_n$, we have $x \in \overline{Z_n \cup P_n} = Z_n \cup P_n$. Then $x \in P_n \setminus Z_{n+1}$.
From (c) and (e) it follows that there is a unique index $t \hat{\,} k \in \omega^{n+1}$ with $x \in \overline{U_{t \hat{\,} k}} \subset P_n$. According to (f), we have 
\[\overline{P_{n+1}} \cap \overline{U_{t \hat{\,} k}} = 
\{x_{t \hat{\,} k} \} \cup \bigcup \left\{ \overline{U_{t \hat{\,} k \hat{\,} i}} \colon i \in \omega \right\}.\]
By (e), $x \in \overline{U_{t \hat{\,} k \hat{\,} j}}$ for some $j$. Then $x \in P_{n+1}$, a contradiction. 

Thus, the set $Z_n \cup P_n$ is closed in $X$ for each $n$.

By construction, $Z \subset Z_n \cup P_n$ for each $n$. Then $\overline{Z} \subset \bigcap \{Z_n \cup P_n \colon n \in \omega \}$.

Striving for a contradiction, suppose there is a point $x \in \overline{Z} \setminus Z$. Fix the least $i$ with $x \in X_i$. Obviously, $x \notin Z_i$. Condition (b) implies that $P_i \cap X_i = \emptyset$. Then $x \notin Z_i \cup P_i$, a contradiction. 
\end{proof}

Let us recall the well-known theorem of van Douwen~\cite{D87}. For a first countable regular space $X$, he proved that every closed subspace of $X$ is a Baire space if and only if $X$ has no countable closed crowded subspace. Here a space is said to be \textit{crowded} if it has no isolated points. Note that up to homeomorphism the space of rational numbers is the only countable first-countable crowded regular space ~\cite{Si}. In other words, for a first countable regular space $X$, van Douwen proved that $X$ is completely Baire if and only if the space of rational numbers does not embed as a closed subspace into $X$.

Now we obtain a generalization of this result.

\begin{thm}\label{t:2}
A hereditary weakly Preiss-Simon regular space $X$ is completely Baire if and only if the space of rational numbers does not embed as a closed subspace into $X$.
\end{thm}

\begin{proof}
Clearly, if $X$ contains a closed copy of the space of rational numbers, then it cannot be completely Baire. For the other direction, aiming for a contradiction, suppose that $X$ is not completely Baire. Then $X$ contains a closed subset $F$ of the first category. It remains to apply Theorem~\ref{t:1} to $F$.
\end{proof}

\begin{lem}\label{L:2}
Let $f \colon X \rightarrow \mathbb{Q}$ be a continuous mapping of a Baire space $X$ onto the space of rational numbers $\mathbb{Q}$.
 
Then there is an open set $U \subset X$ such that $f(U)$ is not resolvable.
\end{lem}

\begin{proof}
Denote by $X^\circ_q$ the interior of $f^{-1}(q)$ for every point $q \in \mathbb{Q}$. Put $\mathbb{Q}^\circ = \{q \in \mathbb{Q} \colon X^\circ_q \neq \emptyset \}$.

Let us verify that the open set $A = \bigcup \{ X^\circ_q \colon q \in \mathbb{Q}^\circ \}$ is dense in $X$. Suppose the contrary and choose a nonempty open set $W \subset X$ missing $A$. For each point $q \in \mathbb{Q} \setminus \mathbb{Q}^\circ $ the intersection $f^{-1}(q) \cap W$ is a closed nowhere dense subset of $W$. For each point $q \in \mathbb{Q}^\circ $ the intersection $f^{-1}(q) \cap W = (f^{-1}(q) \setminus A) \cap W $ is the same. This implies that $W$ is a set of the first category on itself. On the other hand, $W$ is a Baire space as an open subset of the Baire space $X$, a contradiction.

Then $\mathbb{Q}^\circ = f(A)$ is dense in $\mathbb{Q}$ because $f$ is continuous.

Take $B \subset \mathbb{Q}^\circ$ such that $B$ is dense, co-dense in $\mathbb{Q}$. Clearly, $B$ is not a resolvable subset of $\mathbb{Q}$. The open set $U =
\bigcup \{ X^\circ_q \colon q \in B \}$ is as required.
\end{proof}

Now we are ready to give the main result of the paper.

\begin{thm}\label{Th-m}
Let $f \colon X \rightarrow Y$ be a continuous mapping of a completely Baire space $X$ onto a hereditary weakly Preiss-Simon regular space $Y$ such that the image of every open subset of $X$ is a resolvable set in $Y$. Then $Y$ is completely Baire.
\end{thm}

\begin{proof}
Striving for a contradiction, suppose that $Y$ is not completely Baire. By Theorem~\ref{t:2}, $Y$ contains a closed copy $Q$ of the space of rational numbers.
The inverse image $f^{-1}(Q)$ is a Baire space as a closed subset of $X$. By Lemma~\ref{L:2}, there is an open in $f^{-1}(Q)$ set $U \subset f^{-1}(Q)$ such that $f(U)$ is not resolvable. Take an open set $U^* \subset X$ with $U^* \bigcap f^{-1}(Q) = U$. Then $f(U^*)$ is not resolvable because the closed subset $Q \bigcap f(U^*) = f(U)$ of a resolvable set must be resolvable. A contradiction.
\end{proof}

\end{document}